\documentclass[11pt]{amsart}

\usepackage{amssymb, amsmath, amscd, amsthm, graphicx, psfrag,rotating}
\usepackage[matrix,arrow,curve]{xy}
\xyoption{dvips}              
\usepackage{wasysym}
\usepackage{epstopdf}

\usepackage{amssymb, amsmath, amscd, amsthm,
color, epsfig
}
\usepackage[english]{babel}

\title{Two-loop part of the rational homotopy of spaces of long embeddings}
\author{Jim Conant}
\address{University of Tennessee}
\email{jconant@math.utk.edu }
\author{Jean Costello}
\address{University of Minnesota}
\email{coste085@umn.edu}
\author{Victor Turchin}
\address{Kansas State University}
\email{turchin@math.ksu.edu}
\author{Patrick Weed}
\address{University of California, Davis}
\email{psweed@math.ucdavis.edu}

\newtheorem{theorem}{Theorem}[section]
\newtheorem{corollary}[theorem]{Corollary}
\newtheorem{lemma}[theorem]{Lemma}
\newtheorem{definition}[theorem]{Definition}
\newtheorem{remark}[theorem]{Remark}

\newcommand{\otherleavescount}{\sum_{i=1}^{3}\frac{k_i(k_i + 1)}{2}}

\newcommand{\Q}{{\mathbb{Q}}}

\newcommand{\R}{{\mathbb R}}

\newcommand{\calEpi}{{\mathcal E}^{m,N}_\pi}
\newcommand{\calHpi}{{\mathcal H}^{m,N}_\pi}

\newcommand{\Ebarmn}{{\overline{\mathrm{Emb}}}_c(\R^m,\R^N)}
\newcommand{\Embmn}{\mathrm{Emb}_c(\R^m,\R^N)}
\newcommand{\Immmn}{\mathrm{Imm}_c(\R^m,\R^N)}

\newcommand{\im}{\operatorname{im}}

\newcommand{\calA}{\mathcal{A}}
\newcommand{\calB}{\mathcal{B}}
\newcommand{\calC}{\mathcal{C}}

\newcommand{\Sym}{\operatorname{Sym}}
\newcommand{\ASym}{\operatorname{ASym}}
\newcommand{\even}{\mathrm{even}}
\newcommand{\odd}{\mathrm{odd}}

\newcommand\rth{\refstepcounter{equation}}
\newcommand\numb{\rth{\rm \theequation}}
\numberwithin{equation}{section}

\begin{document}

\sloppy

\begin{abstract}
Arone and Turchin defined graph-complexes computing the rational homotopy of the spaces  of long embeddings. The graph-complexes split into a direct sum by the number of loops in graphs. In this paper we compute the homology of its two-loop part.
\end{abstract}

\maketitle

\section{Introduction}\label{s_intro}
Let $\Embmn$ denote the space of smooth embeddings $\R^m\hookrightarrow \R^N$ coinciding with a fixed linear embedding outside a compact subset of $\R^m$. By $\Ebarmn$ we denote the homotopy fiber of the inclusion $\Embmn\hookrightarrow\Immmn$, where $\Immmn$ is the space of immersions with the same behavior at infinity. In~\cite{AT1,AT2} Arone and Turchin study the rational homology and homotopy of $\Embmn$ and $\Ebarmn$. In particular in~\cite{AT2} the authors define graph-complexes $\calEpi$ computing the rational homotopy $\pi_*\Ebarmn\otimes\Q$, $N\geq 2m+2$. The graph-complexes up to a regrading depend only on the parities of $m$ and $N$. Thus there are only 4 cases to consider. The connection between
$\pi_*\Ebarmn\otimes\Q$ and $\pi_*\Embmn\otimes\Q$, $N\geq 2m+2$, is established by~\cite[Corollary~4.3]{AT2}.
The graph-complexes $\calEpi$ split into a direct sum of complexes by the number of univalent vertices (this number is called {\it Hodge degree}) and the number of loops (first Betti number) in the graphs. This paper
computes  the homology of the part of $\calEpi$ spanned by  two-loop graphs. The homology of the part
 spanned by trees and of the part  spanned by one-loop graphs was computed in~\cite[Section~3]{AT2}.

Theorems~\ref{t_homology_m_odd_n_odd}, \ref{t_homology_m_even_n_even}, \ref{t_homology_m_even_n_odd}, \ref{t_homology_m_odd_n_even}
describe generating functions of the homology ranks and of the Euler characteristics in each of the four cases. Theorems~\ref{t_rank_m_odd_n_odd}-\ref{t_rank_m_odd_n_even} give explicit formulae for those ranks. We can make two conclusions from our computations. Firstly, for all parities of $m$ and $N$,  the homology of the two-loop part for any Hodge degree is concentrated in only one grading. Thus the Euler characteristics completely determine the homology ranks. Secondly, the ranks grow linearly with the Hodge degree.

Our method  is borrowed from similar computations of the dimensions of the space of 2-loop uni-trivalent graphs modulo $IHX$ and $AS$ relations~\cite{Nakats}, which appears in the study of finite type knot invariants, see also~\cite{Dasbach,MoskOhtsuki}. The latter space is  the bottom degree part of our homology for the case when both $m$ and $N$ are odd.
The field of coefficients for  the considered complexes  is always $\Q$.


\section{Complexes of uni-$\geq$3-valent graphs}\label{s_graph}
In this section we recall  definition of the complex $\calEpi$ of uni-$\geq$3-valent graphs from~\cite{AT2}. The homology of this complex is naturally isomorphic to~$\pi_*\Ebarmn\otimes\Q$, $N\geq 2m+2$.

The complex $\calEpi$ is spanned by abstract connected graphs having a non-empty set of non-labeled {\it external vertices} of valence~1, and a possibly empty set of non-labeled {\it internal vertices} of valence~$\geq 3$. The graphs are allowed to have edges joining a vertex to itself and multiple edges. For such graph define its {\it orientation set} as the union of the set of its external vertices (considered as elements of degree $-m$), the set of its internal vertices (considered as elements of degree $-N$), and the set of its edges (considered as elements of degree $(N-1)$). By an {\it orientation} of a graph we will understand ordering of its orientation set together with an orientation of all its edges. Two such graphs are {\it equivalent} if there is a bijection between their sets of vertices and edges respecting the adjacency structure of the graphs, orientation of the edges, and the order of the orientation sets. The space of $\calEpi$ is the quotient space of the vector space freely spanned by equivalence classes of such oriented graphs modulo the orientation relations:

\vspace{.2cm}

(1) $\Upsilon_1=(-1)^n\Upsilon_2$ if $\Upsilon_1$ differs from $\Upsilon_2$ by reversing the orientation of an edge.

(2) $\Upsilon_1=\pm\Upsilon_2$, where $\Upsilon_2$ is obtained from $\Upsilon_1$ by a permutation of the orientation set. The sign here is the Koszul sign of the permutation taking into account the degrees of the elements.

\vspace{.2cm}

Notice that if a graph has a symmetry that produces a negative sign, then such a graph is zero modulo orientation relations (1-2).
The total degree of a graph is the sum of degrees of all the elements from its orientation set. The differential $\partial\Upsilon$ of a graph $\Upsilon\in\calEpi$ is defined as the sum of expansions of its internal vertices. An expanded vertex is  replaced  by an edge. The set of edges adjacent to the  expanded vertex splits into two sets -- one containing the edges that go to one vertex of the new edge and the other set containing the edges that go to the other vertex.  An expansion of a vertex of valence $\ell$ is  a sum of $\frac{2^\ell-2\ell-2}{2}=2^{\ell-1}-\ell-1$ graphs obtained in such way. One subtracts  $2\ell+2$ to exclude graphs with internal vertices of valence $<3$, and one divides by 2 because of the symmetry. The orientation set of a new graph is obtained by adding the new vertex and the new edge as the first and second elements to the orientation set, and by orienting the new edge from the old vertex to the new one. There is a freedom which of 2 vertices of the new edge is considered as a new one and which as an old one, but regardless of this choice, the orientation of the boundary graph is the same. All the graphs in the differential appear with positive sign (the sign is hidden in the way we order the orientation set and orient the new edge).

Notice that the differential preserves the number of external vertices $t$ referred as {\it Hodge degree} and also the first Betti number $L$ (number of loops in a graph). The first Betti number $s$ of the graph obtained by gluing together all univalent vertices will be called {\it complexity}. It is also preserved by the differential. One obviously has
$$
L=s-t+1.
$$
The part of $\calEpi$ of Hodge degree $t$ and complexity $s$ is denoted by $\calEpi(s,t)$.

We will also define {\it defect} of a graph as the sum $\sum_v(|v|-3)$, where $v$ runs through the internal vertices of the graph, and $|v|$ is the valence of $v$. The defect measures how much the graph is different from a uni-trivalent one. The differential decreases the defect by one.

\section{Hairy graph-complexes}\label{s_hairy}
In this section we define a quasi-isomorphic subcomplex $\calHpi$ of $\calEpi$ which will require a few definitions.

An edge in a connected graph is called {\it tree-type} if when one removes it, the graph becomes disconnected, whose exactly one connected part is a tree. By the {\it frame} of a graph we will mean its subgraph generated by edges which are not tree-type. Notice that the frame of a connected graph is always connected. A tree-type edge is called a {\it hair} if one of its vertices is univalent and the other one belongs to the frame and has no other tree-type edges adjacent to it. A vertex is called {\it cut} if when one removes it  the graph becomes disconnected.

One defines the {\it complex $\calHpi$ of hairy graphs} as the subcomplex of $\calEpi$ spanned  by trees with $\leq 3$ external vertices and by graphs whose frame has no cut-vertices and whose  tree-type edges are all hairs.

\begin{theorem}\label{t_hairy}
The inclusion $\calHpi\hookrightarrow\calEpi$ is a quasi-isomorphism.
\end{theorem}

\begin{proof}
For the part $\calHpi(t+1,t)\hookrightarrow\calEpi(t+1,t)$, i.e. the part which is spanned by trees,  see~\cite[Proposition~3.2]{AT2}. For the case when the first Betti number is positive the proof follows from the argument of~\cite[Theorem~1.1]{ConVog}. In that paper, graphs  do not have hairs, the orientation data is not as general, and the differential contracts, rather than expands edges. However the argument still works with a little modification. One writes $\partial=\partial_s+\partial_{ns}$, where $\partial_s$ expands cut vertices into separating edges and $\partial_{ns}=\partial-\partial_s$.
Consider the subspace of graphs which have a cut vertex which is not a trivalent vertex incident to a hair, which is a complex with respect to $\partial_s$. By the same argument as in that paper, the homology of this complex is trivial. So in the spectral sequence for the double complex, only graphs without cut vertices survive, and the surviving boundary operator $\partial_{ns}$ is the standard  boundary restricted to the hairy graph subcomplex.
\end{proof}

\section{Two-loop part of $\calHpi$}\label{s_2loops}
From Theorem~\ref{t_hairy} one can immediately conclude that the loopless part of $\calHpi$ is always one-dimensional. It is spanned by the tree with two vertices when $m$ and $N$ are of the same parity and by the tree with three external and one internal vertices when $m$ and $N$ are of opposite parity. The differential is obviously zero. The one-loop part of $\calHpi$ is spanned by wheels -- graphs obtained from a circle by attaching several hairs. \cite[Proposition~3.3]{AT2} tells when such wheels survive their dihedral symmetry. The differential is also zero in this situation.

Let $\calC_i$ denote the space of two-loop hairy graphs of defect $i$. Notice that the frame of any such graph has the shape of the $\Theta$-graph. The defect $i$ can be either 0, 1, or 2, see the figure below.

 \vspace{.3cm}

\begin{center}
\psfrag{C2}[0][0][1][0]{$\calC_2$}
\psfrag{C1}[0][0][1][0]{$\calC_1$}
\psfrag{C0}[0][0][1][0]{$\calC_0$}
\includegraphics[width=14cm]{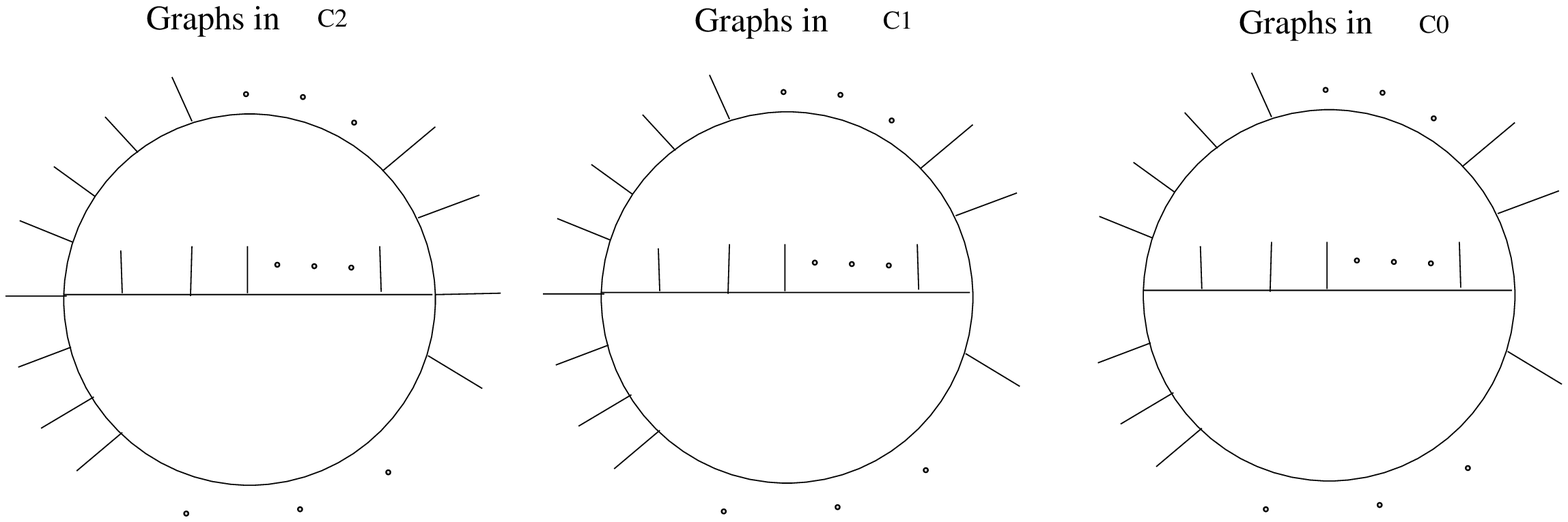}
\end{center}

\vspace{.3cm}

One gets the complex
$$
0 \xrightarrow{\partial _3} \calC_2 \xrightarrow{\partial_2} \calC_1 \xrightarrow{\partial_1}\calC_0 \xrightarrow{\partial_0} 0.
\eqno(\numb)\label{eq:hairy_2loop}
$$
We will see that the differential $\partial_2 \colon \calC_2\to \calC_1$ is injective. Which means that the homology must be concentrated in defect~0 and~1. Let $a_k$, respectively $b_k$, be the rank of the homology of defect~0, respectively~1, in Hodge degree~$k$. Let $\chi_k$ denote the Euler characteristic of~\eqref{eq:hairy_2loop} in Hodge degree $k$. We will compute the generating functions
$$
h_0(t)=\sum_{k=1}^{\infty}a_kt^k,\quad h_1(t)=\sum_{k=1}^\infty b_kt^k,\quad \chi(t)=\sum_{k=1}^{\infty}\chi_kt^k.
$$
Obviously, $h_0$ and $h_1$ depend on the parities of $m$ and $N$. The total degree of a graph with $k$ hairs in $C_0$ is $k(N-m-2)+N-3$. This implies the following lemma.

\begin{lemma}\label{l_euler}
 One has
$$
\chi(t)=(-1)^{N-1}\left[h_0((-1)^{N-m}t)-h_1((-1)^{N-m}t)\right].
$$
\end{lemma}
\begin{proof}
Follows from injectivity of $\partial_2$, see Section~\ref{s_homology}.
\end{proof}

The main idea for our computations is that the spaces of graphs will be encoded as certain spaces of polynomials.
Denote by $x_1^{k_1}x_2^{k_2}x_3^{k_3}$ the graph from $\calC_2$ that has $k_1$ hairs on the upper edge of $\Theta$, $k_2$ hairs on the middle one, and $k_3$ hairs on the lower one. Its edges are oriented as in the figure below.

 \vspace{.3cm}

\begin{center}
\psfrag{A1}[0][0][1][0]{$A_1$}
\psfrag{A2}[0][0][1][0]{$A_2$}
\psfrag{B1}[0][0][1][0]{$B_1$}
\psfrag{B2}[0][0][1][0]{$B_2$}
\psfrag{a1}[0][0][1][0]{$a_1$}
\psfrag{a2}[0][0][1][0]{$a_2$}
\psfrag{b1}[0][0][1][0]{$b_1$}
\psfrag{b2}[0][0][1][0]{$b_2$}
\psfrag{b3}[0][0][1][0]{$b_3$}
\psfrag{C1}[0][0][1][0]{$C_1$}
\psfrag{C2}[0][0][1][0]{$C_2$}
\psfrag{C3}[0][0][1][0]{$C_3$}
\psfrag{c1}[0][0][1][0]{$c_1$}
\psfrag{c2}[0][0][1][0]{$c_2$}
\psfrag{ck1}[0][0][1][0]{$c_{k_1}$}
\psfrag{Ck1}[0][0][1][0]{$C_{k_1}$}
\psfrag{C1p}[0][0][1][0]{$C_1'$}
\psfrag{C2p}[0][0][1][0]{$C_2'$}
\psfrag{Ck1p}[0][0][1][0]{$C_{k_1}'$}
\psfrag{c1p}[0][0][1][0]{$c_1'$}
\psfrag{c2p}[0][0][1][0]{$c_2'$}
\psfrag{ck1p}[0][0][1][0]{$c_{k_1}'$}
\psfrag{D1}[0][0][1][0]{$D_1$}
\psfrag{D2}[0][0][1][0]{$D_2$}
\psfrag{D3}[0][0][1][0]{$D_3$}
\psfrag{d1}[0][0][1][0]{$d_1$}
\psfrag{d2}[0][0][1][0]{$d_2$}
\psfrag{D1p}[0][0][1][0]{$D_1'$}
\psfrag{D2p}[0][0][1][0]{$D_2'$}
\psfrag{d1p}[0][0][1][0]{$d_1'$}
\psfrag{d2p}[0][0][1][0]{$d_2'$}
\psfrag{E1}[0][0][1][0]{$E_1$}
\psfrag{E2}[0][0][1][0]{$E_2$}
\psfrag{Ek3}[0][0][1][0]{$E_{k_3}$}
\psfrag{e1}[0][0][1][0]{$e_1$}
\psfrag{e2}[0][0][1][0]{$e_2$}
\psfrag{E1p}[0][0][1][0]{$E_1'$}
\psfrag{E2p}[0][0][1][0]{$E_2'$}
\psfrag{e1p}[0][0][1][0]{$e_1'$}
\psfrag{e2p}[0][0][1][0]{$e_2'$}
\includegraphics[width=14cm]{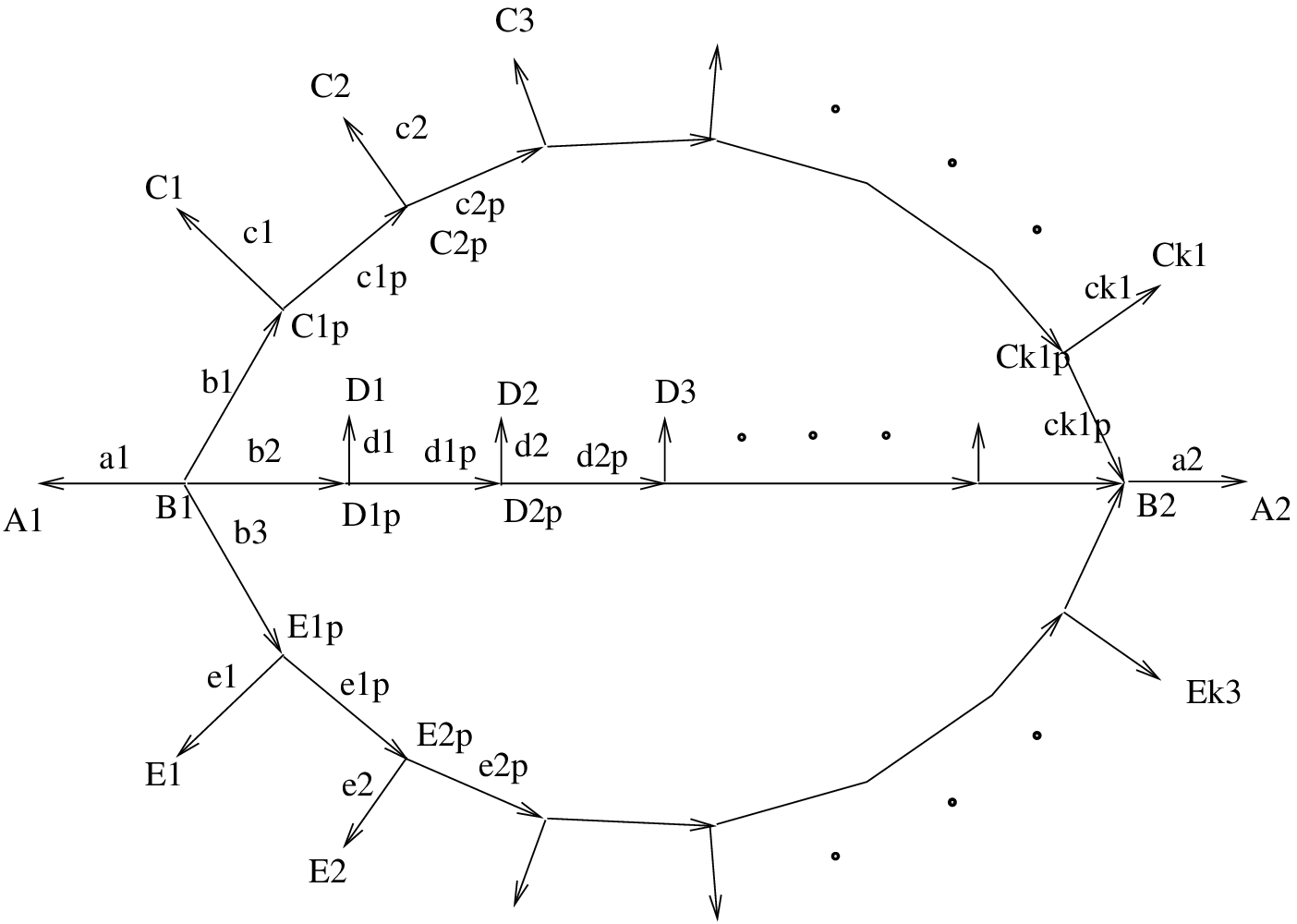}
\end{center}

 The ordering of its orientation set is as follows. First one has $A_1A_2a_1a_2B_1B_2b_1b_2b_3$. Then one puts $(C_1c_1C_1'c_1')(C_2c_2C_2'c_2')\ldots(C_{k_1}c_{k_1}C_{k_1}'c_{k_1}')$, then a similar product of $(D_id_iD_i'd_i')$, $i=1\ldots k_2$, corresponding to the middle edge, then a similar product of $(E_ie_iE_i'e_i')$, $i=1\ldots k_3$.

Abusing notation we will denote by $x_1^{k_1}x_2^{k_2}x_3^{k_3}$ similar graphs in $\calC_1$ and $\calC_0$. The ordering of the orientation sets will start in these cases as $A_1a_1B_1B_2b_1b_2b_3$, $B_1B_2b_1b_2b_3$
respectively.

The two-loop graphs have two possible types of symmetry: symmetry with respect to the vertical line and transpositions of edges in  the $\Theta$-graph.  If a graph has a symmetry that produces a negative sign, it means the graph is zero modulo orientation relations. The lemma below computes the signs arising from these symmetries.

\begin{lemma}\label{l_sign}
(1) The symmetry of the graph $x_1^{k_1}x_2^{k_2}x_3^{k_3}\in C_2$ with respect to the vertical line produces the sign
$$
(-1)^{m+N+1}(-1)^{k_1+k_2+k_3}(-1)^{(m+N)\sum_{i=1}^3\frac{k_i(k_i-1)}2}.
$$

(2) The symmetry of the graph $x_1^{k_1}x_2^{k_2}x_3^{k_3}\in C_0$ with respect to the vertical line produces the sign
$$
(-1)^{k_1+k_2+k_3}(-1)^{(m+N)\sum_{i=1}^3\frac{k_i(k_i-1)}2}.
$$

(3) The graph $x_1^{k_1}x_2^{k_2}x_3^{k_3}\in C_i$, $i=0\ldots 2$, is equal to $x_1^{k_2}x_2^{k_1}x_3^{k_3}$ with the sign
$$
(-1)^{N-1}(-1)^{k_1k_2(m+N)}.
$$
\end{lemma}

\begin{proof}
Notice that the orientation of $x_1^{k_1}x_2^{k_2}x_3^{k_3}$ that we defined matters only for~(3). In the latter case the sign $(-1)^{N-1}$ arises from the transposition of $b_1$ and $b_2$. The sign $(-1)^{k_1k_2(m+N)}$ is due to the permutation of the factors $(D_id_iD_i'd_i')$ with $(E_je_jE_j'e_j')$.

The sign in (1) is the product of
\begin{itemize}
\item $(-1)^m(-1)^{N-1}(-1)^N$ -- corresponding to the transpositions of $A_1$ with $A_2$, $a_1$ with $a_2$, $B_1$ with $B_2$;
\item $(-1)^{(N-1)\sum_{i=1}^3\frac {k_i(k_i+1)}2}$ -- corresponding to permutation of $b_i$'s, $c_i'$'s, $d_i'$'s, $e_i'$'s;
\item $(-1)^{(m+1)\sum_{i=1}^3\frac {k_i(k_i-1)}2}$ -- corresponding to permutation of $C_i$'s, $c_i$'s, $C_i'$'s, $D_i$'s, $d_i$'s, $D_i'$'s, $E_i$'s, $e_i$'s, $E_i'$'s;
\item $(-1)^{N(k_1+k_2+k_3+3)}$ -- corresponding to the change of orientation of edges.
\end{itemize}

The sign in (2) is the same as in (1) except that one should not count the transpositions of $A_1$ with $A_2$ and of $a_1$ with $a_2$.
\end{proof}

\section{two-loop homology computations}\label{s_homology}

\subsection{When both $N$ and $m$ are odd}
Let $\Q[x_1,x_2,x_3]$ denote the the free commutative algebra generated by $x_1$, $x_2$, $x_3$.

\begin{definition}
We say that a  polynomial  is even (odd, respectively) if all of its monomials are of even (odd, respectively) degree.
\end{definition}

In the previous section we encoded two-loop  graphs by monomials $x_1^{k_1}x_2^{k_2}x_3^{k_3}$. The symmetric group $S_3$ acts on the space $\Q[x_1,x_2,x_3]$ by permuting the variables $x_1$, $x_2$, $x_3$, which geometrically correspond to the reordering of the edges in the theta-graph. Lemma~\ref{l_sign}~(3) tells us what sign should be used to obtain an equivalent monomial-graph. Instead of taking a quotient by this action we will be taking invariants since in characteristic zero, invariants coincide with coinvariants.
When both $m$ and $N$ are odd, a graph with $k_i$, $i=1\ldots 3$, hairs respectively, will be denoted instead by $\frac 16\sum_{\sigma\in S_3}x_{\sigma(1)}^{k_1}x_{\sigma(2)}^{k_2}x_{\sigma(3)}^{k_3}$. It follows from  Lemma~\ref{l_sign} that the space $\calC_2$ in this case is represented by odd symmetric polynomials, the space $\calC_1$ is represented by the space of all symmetric polynomials, and the space $\calC_0$ is  represented by even symmetric polynomials whose all monomials are of strictly positive  degree:
$$
0 \xrightarrow{\partial _3} \Sym^{\mathrm{\odd}}[x_1, x_2, x_3] \xrightarrow{\partial_2} \Sym[x_1, x_2, x_3] \xrightarrow{\partial_1} \Sym^{\mathrm{\even}}_{>0}[x_1, x_2, x_3] \xrightarrow{\partial_0} 0.
$$


The following result is well known.

\begin{theorem}
The algebra $\Sym[x_1, x_2, x_3]$ is isomorphic to the free commutative algebra generated by three letters: $\Q[e_1, e_2, e_3]$, where  $e_1 = x_1 + x_2 + x_3$, $e_2 = x_1x_2 + x_1x_3 + x_2x_3$, and $e_3 = x_1x_2x_3$.
\end{theorem}

So, we have as a basis for $\Sym[x_1, x_2, x_3]$ polynomials of the form $e_1^\alpha e_2^\beta e_3^\gamma$, where $\alpha , \beta , \gamma \in \mathbb{N}_0$. Also, every such symmetric polynomial has degree $\alpha + 2\beta + 3\gamma$, since $e_1$ is of degree 1, $e_2$ is of degree 2, and $e_3$ is of degree 3. Our next corollary follows, which we use freely throughout this section.

\begin{corollary}
$$
\Sym^{\mathrm{\even}}[x_1, x_2, x_3] = \langle e_1^\alpha e_2^\beta e_3^\gamma |\, \alpha + \gamma \equiv 0 \,\mathrm{mod}\, 2\rangle
$$
and
$$
\Sym^{\mathrm{\odd}}[x_1, x_2, x_3] = \langle e_1^\alpha e_2^\beta e_3^\gamma |\, \alpha + \gamma \equiv 1 \,\mathrm{mod}\, 2\rangle.
$$
\end{corollary}

Now we will describe the differentials. Let $A\colon\Q[x_1,x_2,x_3]\to \Q[x_1,x_2,x_3]$ be the automorphism of $\Q[x_1,x_2,x_3]$ sending each generator $x_i$ to $-x_i$, $i=1\ldots 3$. Geometrically this automorphism corresponds to the symmetry with respect to the vertical line that our hairy $\Theta$-graphs have. Notice that the projection to the space of even, respectively odd, monomials is described by $f\mapsto \frac 12(f+A(f))$, respectively $f\mapsto \frac 12 (f-A(f))$. Sometimes the even, respectively odd, part of $f$ will also be denoted by $[f]_{\even}$, respectively $[f]_{\odd}$.

\begin{lemma}\label{l_diff_m_odd_n_odd}
In the case both $m$ and $N$ are odd, the differentials of the hairy two-loop graph-complex are defined as follows
\begin{align*}
\partial_2f&=-2e_1f,\\
\partial_1f&=-\frac 12(e_1f+A(e_1f))=-[e_1f]_{\even}.
\end{align*}
\end{lemma}

\begin{proof}
The coefficient 2 in the definition of $\partial_2$ is due to the fact that there are two 4-valent vertices to expand, which produce the same result due to the symmetry. Such an expansion produces a new hair on one of the three edges of $\Theta$ which is algebraicly expressed as a multiplication by
$e_1=x_1+x_2+x_3$. The sign minus\footnote{These signs as well as the coefficient 2 are not important for the homology computations. For this reason we do not show how these signs are calculated.} in both
cases is due to the rule of signs for the differential, see Section~\ref{s_graph}. See also
Remark~\ref{r_diff_general} where general formulas for the differentials in all the four cases are given.
\end{proof}

We now wish to compute homology. Notice
$$
H_i := \frac{\ker (\partial_i)}{\im (\partial_{i+1})}.
$$
Since $\partial_0 : C_0 \rightarrow 0$,
$$
 ker(\partial _0) = \langle e_1 ^\alpha e_2 ^\beta e_3 ^\gamma |\, \alpha + \gamma \equiv 0 \,\mathrm{mod}\, 2,\, \alpha+\beta +\gamma>0\rangle.
$$
Secondly,  we have
$$
\im(\partial _1) = \langle e_1 ^{\alpha + 1} e_2 ^\beta e_3 ^\gamma |\,  \alpha + \gamma \equiv 1  \,\mathrm{mod}\, 2\rangle.
$$
So, we have
\begin{align*}
H_0  &= \frac{\ker(\partial _0)}{\im(\partial _1)} \\
&= \frac{\langle e_1 ^ \alpha e_2 ^\beta e_3 ^\gamma |\, \alpha + \gamma \equiv 0 \,\mathrm{mod}\, 2,\, \alpha+\beta+\gamma>0\rangle}{\langle e_1 ^{\alpha + 1} e_2 ^\beta e_3 ^ \gamma |\, \alpha + \gamma \equiv 1 \,\mathrm{mod}\, 2\rangle} \\
&= \langle e_2 ^\beta e_3 ^ \gamma |\, \gamma \equiv 0 \,\mathrm{mod}\, 2,\, \beta+\gamma>0\rangle.
\end{align*}

We will now compute $H_1$. We have
$$
ker(\partial_1) = \langle e_1 ^ \alpha e_2 ^\beta e_3 ^\gamma |\, \alpha + \gamma \equiv 0 \,\mathrm{mod}\, 2\rangle.
$$
In other words this kernel consists of symmetric even polynomials.
Similarly we have
$$
\im(\partial_2) =  \langle e_1 ^ {\alpha + 1} e_2 ^\beta e_3 ^\gamma |\, \alpha + \gamma \equiv 1 \,\mathrm{mod}\, 2\rangle.
$$
Thus,
\begin{align*}
H_1 &= \frac{\langle e_1 ^ \alpha e_2 ^\beta e_3 ^\gamma |\, \alpha + \gamma \equiv 0\ \,\mathrm{mod}\,\ 2\rangle}{\langle e_1 ^{\alpha + 1} e_2 ^\beta e_3 ^ \gamma |\, \alpha + \gamma \equiv 1 \,\mathrm{mod}\,\ 2\rangle} \\
&= \langle e_2 ^\beta e_3 ^ \gamma |\, \gamma \equiv 0 \,\mathrm{mod}\, 2 \rangle.
\end{align*}


We see that $H_0 = \mathbb{Q}_{>0} [e_2, e_3 ^2]$ and $H_1 = \mathbb{Q} [e_2, e_3 ^2]$. As a consequence we get the following result.

\begin{theorem}\label{t_homology_m_odd_n_odd}
In case both $m$ and $N$ are odd, the generating functions of the  dimensions of the two-loop homology $H_0$ and $H_1$, and of the Euler characteristics,  are
\begin{align*}
h_0(t)& = \frac{1}{(1 - t^2)(1 - t^6)} - 1,\\
h_1(t)& = \frac{t}{(1 - t^2)(1 - t^6)},\\
\chi (t)& = \frac{1}{(1 + t)(1 - t^6)} - 1.
\end{align*}
\end{theorem}
Although the underlying spaces of $H_0$ and $H_1$ are the same, we add the factor $t$ in $h_1(t)$  because of the additional hair growing from the left vertex of the $\Theta$-graph. The formula for $\chi(t)$  was obtained using Lemma~\ref{l_euler}.


\subsection{When both $N$ and $m$ are even}
This is the second case where $m$ and $N$ have the same parity.  We consider the same space of polynomials $\Q[x_1,x_2,x_3]$ as in our previous case.  However, we obtain a negative sign  when we transpose two edges, see Lemma~\ref{l_sign}~(3). Thus, we can no longer use symmetric polynomials to describe our space; instead, we turn to antisymmetric polynomials.  An antisymmetric polynomial is defined as a polynomial on $n$ variables such that any transposition of variables results in a negative sign.
The two-loop hairy graph-complex has now the following form:
$$
0 \xrightarrow{\partial_3} \ASym^{\odd}[x_1,x_2,x_3] \xrightarrow{\partial_2} \ASym[x_1,x_2,x_3] \xrightarrow{\partial_1} \ASym^{\even}[x_1,x_2,x_3] \xrightarrow{\partial_0} 0
$$

\begin{lemma}\label{l_diff_m_even_n_even}
In the case both $m$ and $N$ are even the differentials in the two-loop hairy graph-complex are defined as follows
\begin{align*}
\partial_2f&=2e_1f,\\
\partial_1f&=\frac 12(e_1f+A(e_1f))=[e_1f]_{\even}.
\end{align*}
\end{lemma}
See Remark~\ref{r_diff_general} for a general formula.
One can clearly see that $\partial^2=0$.

The following result is well known.

\begin{theorem}
The space $\ASym[x_1,x_2,x_3]$ of antisymmetric polynomials is a free module over the algebra $\Sym[x_1,x_2,x_3]$ of symmetric polynomials generated by the only element
$$
\Delta=(x_1-x_2)(x_1-x_3)(x_2-x_3).
$$
\end{theorem}

%
%
%
%

The terms $\calC_i$ of our complex are generated by the bases:
\begin{align*}
\calC_2&=\langle \Delta e_1^\alpha e_2^\beta e_3^\gamma|\, \alpha+\gamma \equiv 0 \,\mathrm{mod}\, 2 \rangle;\\
\calC_1&=\langle \Delta e_1^\alpha e_2^\beta e_3^\gamma \rangle;\\
\calC_0&=\langle \Delta e_1^\alpha e_2^\beta e_3^\gamma|\, \alpha+\gamma \equiv 1 \,\mathrm{mod}\, 2 \rangle.
\end{align*}

We then compute $H_0$ and $H_1$:
\begin{align*}
H_0 &=\frac{\ker(\partial_0)}{\im (\partial_1)}\\
&=\frac{\langle \Delta e_1^\alpha e_2^\beta e_3^\gamma|\, \alpha+\gamma \equiv 1 \,\mathrm{mod}\, 2 \rangle}{\langle \Delta e_1^{\alpha +1}e_2^\beta e_3^\gamma|\, \alpha+\gamma \equiv 0 \,\mathrm{mod}\, 2 \rangle}\\
&=\langle \Delta e_2^\beta e_3^\gamma|\, \gamma \equiv 1 \,\mathrm{mod}\, 2 \rangle
\end{align*}



Following a similar process for $H_1$:
\begin{align*}
H_1 & =\frac{\ker(\partial_1)}{\im(\partial_2)}\\
&=\frac{\langle \Delta e_1^\alpha e_2^\beta e_3^\gamma|\, \alpha+\gamma \equiv 1 \,\mathrm{mod}\,2 \rangle}{\langle \Delta e_1^{\alpha+1}e_2^\beta e_3^\gamma|\, \alpha+\gamma \equiv 0\,\mathrm{mod}\,2\rangle}\\
&=\langle \Delta e_2^\beta e_3^\gamma|\, \gamma\equiv 1 \,\mathrm{mod}\, 2\rangle
\end{align*}

Thus we immediately obtain the following
\begin{theorem}\label{t_homology_m_even_n_even}
In case both $m$ and $N$ are even, the generating functions of the  dimensions of the two-loop homology $H_0$ and $H_1$, and of the Euler characteristics,  are
\begin{align*}
h_0(t) &= \frac{t^6}{(1 - t^2)(1 - t^6)},\\
h_1(t) &= \frac{t^7}{(1 - t^2)(1 - t^6)},\\
\chi (t) &= -\frac{t^6}{(1 + t)(1 - t^6)}.
\end{align*}
\end{theorem}
\begin{proof}
Notice that both $H_0$ and $H_1$ are free modules over $\Q[e_2,e_3^2]$ generated by $\Delta e_3$. The degree of $\Delta e_3$ is 6 and the graded dimension of   $\Q[e_2,e_3^2]$ is $\frac 1{(1-t^2)(1-t^6)}$. The result follows.
Again to compute $\chi(t)$ we used Lemma~\ref{l_euler}.
\end{proof}

%
%


\subsection{When $N$ is odd and $m$ is even}
When $m$ and $N$ are of different parities, Lemma~\ref{l_sign}~(3) tells us that the generators $x_i$ should be odd to make the action of $S_3$ well defined. To emphasize that they are odd we will denote them by $\xi_i$, $i=1\ldots 3$ instead. Let $\Q\langle \xi_1,\xi_2,\xi_3\rangle$ be a super graded algebra generated by odd $\xi_i$'s, $i=1\ldots 3$, that anti-commute with each other, but not with themselves. To be precise we require $\xi_i\xi_j=-\xi_j\xi_i$ when $i\neq j$, but not  $\xi_i^2=0$. Notice that this algebra is neither commutative nor supercommutative. The group $S_3$ acts on it by renumbering  the variables. We will denote by $\Sym\langle\xi_1,\xi_2,\xi_3\rangle$ its $S_3$-invariant part. Notice however that a symmetrization of a monomial $\xi_1^{k_1}\xi_2^{k_2}\xi_3^{k_3}$ can be zero if for example $k_1=k_2$ and they are odd. Again this algebra is neither commutative nor supercommutative. It is not hard to find its set of generators and relations, but we omit it here since it unfortunately does not help with our homology computations. We will still denote by $e_1:=\xi_1+\xi_2+\xi_3$.

To describe the symmetry with respect to the vertical line and applying the signs from Lemma~\ref{l_sign}~(1-2), we define a linear map
$
A\colon\Q\langle \xi_1,\xi_2,\xi_3\rangle\to\Q\langle \xi_1,\xi_2,\xi_3\rangle
$
by
$$
A(\xi_1^{k_1}\xi_2^{k_2}\xi_3^{k_3})= (-1)^{\otherleavescount}\xi_1^{k_1}\xi_2^{k_2}\xi_3^{k_3}.
$$
Given a super graded algebra $\calA$ one can define its opposite algebra $\calA^{op}$ by taking the same space and defining the new product $\star$ as follows
$$
a\star b:=(-1)^{|a||b|}ba.
$$
A homomorphism $\calA\to \calB^{op}$ is called an antihomomorphism, and an isomorphism $\calA\to\calA^{op}$ is called an antiautomorphism. The following is easy to check.

\begin{lemma}\label{l_anti_aut}
The linear map $A\colon\Q\langle \xi_1,\xi_2,\xi_3\rangle\to\Q\langle \xi_1,\xi_2,\xi_3\rangle$ can also be described as the antiautomorphism of $\Q\langle \xi_1,\xi_2,\xi_3\rangle$ that sends the generators $\xi_i$ to $-\xi_i$, $i=1\ldots 3$.
\end{lemma}

Since $A^2=\mathrm{id}$ this antiautomorphism is an {\it anti-involution}.


\begin{definition}\label{d_*parity}
We say that an element   $f\in\Q\langle \xi_1,\xi_2,\xi_3\rangle$ is   $*$even ($*$odd, respectively) if
$A(f)=f$ ($A(f)=-f$, respectively).
\end{definition}
With this, we may describe our graph complex for this case:
$$
0 \xrightarrow {\partial _3} \Sym^{*\even}\langle\xi_1, \xi_2, \xi_3\rangle \xrightarrow{\partial_2} \Sym\langle \xi_1, \xi_2, \xi_3\rangle \xrightarrow{\partial_1} \Sym^{*\even}_{>0}\langle \xi_1, \xi_2, \xi_3\rangle\xrightarrow{\partial_0} 0
$$
The superscript $*\even$ says that we take only subspace of $*$even polynomials; the subscript $>0$ means that we take polynomials without constant term.

\begin{lemma}\label{l_diff_m_even_n_odd}
For the case $m$ even and $N$ odd, the differentials in the two-loop hairy complex are described as follows:
\begin{align*}
\partial_2f&=-(-1)^{|f|}2fe_1,\\
\partial_1f&=\frac 12(e_1f+A(e_1f))=\frac 12(e_1f-(-1)^{|f|}A(f)e_1)=[e_1f]_{*even}.
\end{align*}
\end{lemma}

\begin{remark}\label{r_diff_general}
In all the four cases the differentials in the two loop hairy graph-complex can be described as follows
\begin{align*}
\partial_2f&=(-1)^{N+|f|}2fe_1;\\
\partial_1f&=(-1)^{mN}\frac 12(e_1f+A(e_1f)).
\end{align*}
In the above $|f|$ refers to the supergrading of $f$. In the case $m$ and $N$ of the same parity, $|f|$ is always even.
\end{remark}

The signs $(-1)^N$ and $(-1)^{mN}$ in the aforementioned formulas are not important for the homology computations for which reason we don't show  how they were calculated.

%
%

The following lemma proves that $\partial^2=0$.

\begin{lemma}\label{l_square}
Let $\omega=\alpha\xi_1+\beta\xi_2+\gamma\xi_3$ be a linear element of $\Q\langle \xi_1,\xi_2,\xi_3\rangle$. If $f\in\Q\langle \xi_1,\xi_2,\xi_3\rangle$ is $*$even, respectively $*$odd, then $\omega f\omega$ is $*$odd, respectively $*$even.
\end{lemma}
\begin{proof}
Using the fact that $A$ is an antiautomorphism and that $A(\omega)=-\omega$, we get
$$
A(\omega f\omega)= (-1)^{|f|+1}A(f\omega)A(\omega)=(-1)^{|f|+1}(-1)^{|f|}A(\omega)A(f)A(\omega)=
-\omega A(f)\omega.
$$
The result follows.
\end{proof}

Unfortunately the property of being $*$even or $*$odd is not preserved by a product. Thus in particular $\Sym^{*\even}\langle\xi_1,\xi_2,\xi_3\rangle$ is not an algebra. This makes our computations  more difficult for the case when $m+N$ is odd. In order to compute the homology we will introduce a good basis in our spaces $\calC_i$, $i=0\ldots 2$. First denote by $(k_1,k_2,k_3)$ the symmetrization of $\xi_1^{k_1}\xi_2^{k_2}\xi_3^{k_3}$ that has coefficient~1 in front of $\xi_1^{k_1}\xi_2^{k_2}\xi_3^{k_3}$. For example,
$$
(1,0,0)=\xi_1+\xi_2+\xi_3;\qquad (2,2,2)=\xi_1^2\xi_2^2\xi_3^2.
$$
The elements $(k_1,k_2,k_3)$ satisfying $k_1\geq k_2\geq k_3$ and  if $k_i=k_{i+1}$, then $k_i$ is even, $i=1,2$, form a basis of $\Sym\langle \xi_1, \xi_2, \xi_3\rangle$ that will be denoted by $\calB$. This set comes with a natural order -- first one compares the degree $k_1+k_2+k_3$, then within the same degree one orders the elements lexicographically. Since the above elements form a basis, any symmetric polynomial $f$ is a unique linear combination of them. The maximal element of $\calB$ that has a non-zero coefficient in the above sum, will be called the leading term of $f$. The subset of $*$even, respectively $*$odd, elements in $\calB$ will be denoted by $\calB^{*\even}$, respectively $\calB^{*\odd}$. For a basis of $\calC_2$ we will simply choose the set $\calB^{*\even}$. For a basis of $\calC_1$ we will take a union of the following sets
\begin{gather*}
\{fe_1|\, f\in\calB\},\\
\{(k_2 + 1, k_2, k_3) |\, k_2 \equiv 1 \,\mathrm{mod}\, 2,\, k_2>k_3\},\\
\{(k_2, k_2, k_3) |\, k_2 \equiv 0 \,\mathrm{mod}\,\ 2, k_2\geq k_3\}.
\end{gather*}
The above union of sets forms a basis of $\calC_1=\Sym\langle \xi_1, \xi_2, \xi_3\rangle$ since every element of $\calB$ appears exactly once as a leading term of one of the elements in the  union.

Since $\partial _2$ up to a sign is a multiplication by $2e_1$ on the right, we have
$$
\im(\partial _2) = \langle f e_1 | f \in \calB^{*\even}\rangle.
$$
Thus $\im(\partial _2)$ is a space spanned by a subset of basis elements.


Now, we compute $\partial _1$ on the basis elements. Applying Lemmas~\ref{l_diff_m_even_n_odd},~\ref{l_square},
\begin{align*}
\partial_1(fe_1)=0,\text{ for }f\in\calB^{*\even};\\
\partial_1(fe_1)=e_1fe_1,\text{ for }f\in\calB^{*\odd}.\\
\end{align*}

Notice that the second line produces elements whose leading terms form the subset of $\calB^{*\even}$ of elements $(k_1,k_2,k_3)$ such that $k_1>k_2+2$, or $k_1=k_2+2$ and $k_2$ is even.

We continue our computation of $\partial_1$ on the basis elements:
\begin{center}
\begin{tabular}{l|l|l}
$X$&$\partial_1 X$&\\
\hline
$(k_2 + 1, k_2, k_3)$ & $2(k_2 + 1, k_2 + 1, k_3)$& $ k_2\text{ odd},\, k_3 \equiv 0\, \mathrm{mod}\, 4,\,  k_2>k_3$\\
$(k_2 + 1, k_2, k_3)$ & $(k_2 + 2, k_2, k_3)$ & $k_2\text{ odd},\, k_3 \equiv 1 \,\mathrm{mod}\, 4,\,  k_2>k_3$\\
$(k_2 + 1, k_2, k_3)$ &  $(k_2 + 2, k_2, k_3)  - (k_2 + 1, k_2, k_3 + 1)$ & $k_2\text{ odd},\, k_3 \equiv 2 \,\mathrm{mod}\, 4,\, k_2>k_3$\\
$(k_2 + 1, k_2, k_3)$ & $2(k_2 + 1, k_2 + 1, k_3)  -(k_2 + 1, k_2, k_3 + 1)$ & $k_2\text{ odd},\, k_3 \equiv 3 \,\mathrm{mod}\, 4,\, k_2>k_3$\\
$(k_2, k_2, k_3)$ & $0$ & $k_2\text{ even},\, k_3 \equiv 0 \,\mathrm{mod}\, 4,\, k_2\geq k_3$\\
$(k_2, k_2, k_3)$ & $(k_2 + 1, k_2, k_3)$ & $k_2\text{ even},\, k_3 \equiv 1 \,\mathrm{mod}\, 4 ,\, k_2>k_3$\\
$(k_2, k_2, k_3)$ & $(k_2 + 1, k_2, k_3)$ & $k_2\text{ even},\, k_3 \equiv 2 \,\mathrm{mod}\, 4  ,\, k_2\geq k_3$\\
$(k_2, k_2, k_3)$ & $(k_2, k_2, k_3 + 1)$ & $k_2\text{ even},\, k_3 \equiv 3 \,\mathrm{mod}\, 4,\, k_2>k_3+1$\\
$(k_3+1, k_3+1, k_3)$ & $3(k_3+1, k_3+1, k_3 + 1)$ & $k_3 \equiv 3 \,\mathrm{mod}\, 4$\\
\end{tabular}
\end{center}
It is easy to see that the images from the first line are twice the images of the next to the last one. If one excludes the elements $\{(k_2 + 1, k_2, k_3)|, k_2 \equiv 1\, \mathrm{mod}\, 2,\, k_3 \equiv 0 \,\mathrm{mod}\, 4,\, k_2>k_3\}$  from the basis of $\calC_1$ (i.e. elements that appear on the first line above), then all the non-zero images of such set have different leading terms. Thus these images are linearly independent. To complete the obtained set to a basis of $\calC_0$ we add elements
$$
\{(k_2 + 1, k_2, k_3)|\, k_2 \equiv 1 \,\mathrm{mod}\, 2,\, k_3 \equiv 0,3 \,\mathrm{mod}\, 4,\, k_2>k_3\}
\cup \{(k_3 + 1, k_3 + 1, k_3)|\, k_3 \equiv 3 \,\mathrm{mod}\, 4\}.
$$
This immediately implies that the above set is a basis of cycles in $H_0$. In other words,
\begin{multline}
H_0= \frac{\ker(\partial _0)}{\im(\partial _1)}=\\
\langle(k_2 + 1, k_2, k_3)|\, k_2 \equiv 1 \,\mathrm{mod}\, 2,\, k_3 \equiv 0,3 \,\mathrm{mod}\, 4,\, k_2>k_3\rangle \oplus \langle(k_3 + 1, k_3 + 1, k_3)|\, k_3 \equiv 3 \,\mathrm{mod}\, 4\rangle.
\label{eq_H0_m_even_n_odd}
\end{multline}
We will now compute $H_1$. It is also easy to see from above that
\begin{multline*}
\ker(\partial_1)=\im(\partial_2)\oplus \langle (k_2, k_2, k_3) |\, k_2 \equiv 0 \,\mathrm{mod}\,\ 2,\, k_3 \equiv 0 \,\mathrm{mod}\, 4,\, k_2\geq k_3\rangle \oplus\\
\langle 2(k_2+1,k_2+1,k_3)-(k_2+1,k_2,k_3+1)|\, k_2\equiv 1\,\mathrm{mod}\, 2,\, k_3\equiv 3\,\mathrm{mod}\, 4,\, k_2>k_3\rangle.
\end{multline*}
Thus we get
\begin{multline}
H_1=\langle (k_2, k_2, k_3) |\, k_2 \equiv 0 \,\mathrm{mod}\, 2,\, k_3 \equiv 0 \,\mathrm{mod}\, 4,\, k_2\geq k_3\rangle \oplus\\
\langle (k_2+1,k_2+1,k_3)-(k_2+1,k_2,k_3+1)|\, k_2\equiv 1\,\mathrm{mod}\, 2,\, k_3\equiv 3\,\mathrm{mod}\, 4,\, k_2>k_3\rangle.
\label{eq_H1_m_even_n_odd}
\end{multline}


\begin{theorem}\label{t_homology_m_even_n_odd}
In case  $m$ is even  and $N$ is odd, the generating functions of the  dimensions of the two-loop graph-homology $H_0$ and $H_1$, and of the Euler characteristics,  are
\begin{align*}
h_0(t) &= \frac{t^3 + t^{11} + t^{14} - t^{15}}{(1-t^4)(1-t^{12})},\\
h_1(t) &= \frac{t + t^{16}}{(1-t^4)(1-t^{12})},\\
\chi (t) &= \frac{t-t^{11}-t^{13}+t^{14}}{(1 + t^2)(1 - t^{12})}.
\end{align*}
\end{theorem}
\begin{proof}
We start with $h_0$. The first summand of~\eqref{eq_H0_m_even_n_odd} has the generating functions of dimensions
$$
\frac{t^3+t^{14}}{(1-t^4)(1-t^{12})}.
$$
In this formula $t^3$ and $t^{14}$ in the numerator correspond to $(2,1,0)$ and $(6,5,3)$ respectively. Then we can simultaneously increase $k_2$ and $k_3$ by 4, or just increase $k_2$ by 2. The first action gives the factor
$$
1+t^{12}+t^{24}+t^{36}+\ldots=\frac 1{1-t^{12}}.
$$
to our generating function.
The second action  gives the factor
$$
1+t^4+t^8+\ldots=\frac 1{1-t^4}.
$$
Similarly the second summand of~\eqref{eq_H0_m_even_n_odd} has the generating function of dimensions
$$
\frac{t^{11}}{1-t^{12}},
$$
where $t^{11}$ corresponds to $(4,4,3)$.

For $h_1(t)$, the first summand of~\eqref{eq_H1_m_even_n_odd} has the generating function of dimensions
$$
\frac t{(1-t^4)(1-t^{12})}.
$$
Here the numerator $t$ corresponds to $(0,0,0)$. Notice that the corresponding graph has exactly~1 hair growing from the left vertex of the $\Theta$ graph. For the second sum we get the function
$$
\frac{t^{16}}{(1-t^4)(1-t^{12})}.
$$
Here $t^{16}$ states for the element $2(6,6,3)-(6,5,4)$.

To compute $\chi(t)$, again we use Lemma~\ref{l_euler}.
\end{proof}


\subsection{When $N$ is even and $m$ is odd}
%

%

%

This case is similar to the one considered in the previous subsection. From Lemma~\ref{l_sign} one can see that our graph-complex  in this situation can be described as follows:
$$
0 \xrightarrow{\partial_3} \ASym^{*\even}\langle\xi_1,\xi_2,\xi_3\rangle\xrightarrow{\partial_2}\ASym\langle\xi_1,\xi_2,\xi_3\rangle
\xrightarrow{\partial_1}\ASym^{*\even}\langle\xi_1,\xi_2,\xi_3\rangle\xrightarrow{\partial_0}0.
$$
In the above $\ASym$ denotes the space of anti-symmetric polynomials in $\Q\langle\xi_1,\xi_2,\xi_3\rangle$. The superscript $*even$ indicates that we only take the subspace of $*$even ones, see Definition~\ref{d_*parity}. The subscript $>0$ indicates that the generating monomials should be of strictly positive degree. The space of anti-symmetric polynomials $\ASym\langle\xi_1,\xi_2,\xi_3\rangle$ is a bimodule over the algebra $Sym\langle\xi_1,\xi_2,\xi_3\rangle$. It is generated by
$$
\Delta_2=\xi_1\xi_2+\xi_2\xi_3+\xi_3\xi_1,
$$
and
$$
\Delta_3=\xi_1\xi_2\xi_3
$$
as a left (or right) module (and by $\Delta_2$ as a bimodule). But contrary to the classical case this action is not free.

\begin{lemma}\label{l_diff_m_odd_n_even}
For the case $m$ even and $N$ odd, the differentials in the two-loop hairy complex are described as follows:
\begin{align*}
\partial_2f&=(-1)^{|f|}2fe_1,\\
\partial_1f&=\frac 12(e_1f+A(e_1f))=\frac 12(e_1f-(-1)^{|f|}A(f)e_1)=[e_1f]_{*even}.
\end{align*}
\end{lemma}

See Remark~\ref{r_diff_general} for a general rule of signs.

Now that we have fully described our complex, we may compute $H_1$ and $H_0$.  We do this in a similar manner to the previous case: we construct a basis for each space by first defining the image of our map $\partial$, then adding in  other elements in our space in a way that each possible leading term appears exactly once. For a monomial $\xi_1^{k_1}\xi_2^{k_2}\xi_3^{k_3}$ its anti-symmetrization with coefficient one in front of $\xi_1^{k_1}\xi_2^{k_2}\xi_3^{k_3}$ will be denoted $[k_1,k_2,k_3]$. For example $\Delta_2=[1,1,0]$, $\Delta_3=[1,1,1]$.  The set
$$
\calB=\{[k_1,k_2,k_3]|\, k_1\geq k_2\geq k_3,\, \text{if $k_i=k_{i+1}$, then $k_i$ is odd, $i=1,2$}\}
$$
is a basis of  $\ASym\langle\xi_1,\xi_2,\xi_3\rangle$. Its subset consisting of $*$even monomials will be denoted by $\calB^{*\even}$. The latter set will be used as a basis for $\calC_2$. For a basis for $\calC_1$ we will take a union of the following sets
\begin{gather*}
\{fe_1 | f \in \calB\}, \\
\{[k_2 + 1, k_2, k_3] \, |\, k_2 \equiv 0 \bmod 2,\, k_2>k_3\}, \\
\{[k_2, k_2, k_3] \,|\, k_2 \equiv 1 \bmod 2,\, k_2\geq k_3\}.
\end{gather*}
\indent Since $\partial_2$ is up to a sign multiplication by $2e_1$ on the right, we have
$$
\im(\partial_2) = \langle fe_1 | f \in \calB^{*\even} \rangle.
$$
Thus, $\im(\partial_2)$ is a space spanned by a subset of basis elements. \\
\indent Now, we compute $\partial_1$ on the basis elements. Applying Lemma 5.15,
\begin{align*}
\partial_1(fe_1) = 0 \text{ for } f \in \calB^{*\even}; \\
\partial_1(fe_1) = e_1fe_1 \text{ for } f \in \calB^{*\odd}.
\end{align*}
We continue our computation of $\partial_1$ on the basis elements:
\begin{center}
\begin{tabular}{l|l|l}
$X$&$\partial_1(X)$&\\
\hline
$[k_2+1,k_2,k_3]$ & $-2[k_2+1,k_2+1,k_3]-[k_2+1,k_2,k_3+1]$ & $k_2\text{ even}, k_3\equiv0\bmod4, k_2 > k_3$ \\
$[k_2+1, k_2, k_3]$ & $[k_2+2, k_2, k_3]-[k_2+1, k_2, k_3+1]$ & $k_2\text{ even}, k_3\equiv1\bmod4, k_2 > k_3$ \\
$[k_2+1, k_2, k_3]$ & $[k_2+2,k_2,k_3]$ & $k_2\text{ even}, k_3\equiv2\bmod4, k_2 > k_3$ \\
$[k_2+1, k_2, k_3]$ & $2[k_2+1,k_2+1,k_3]$ & $k_2\text{ even},k_3\equiv3\bmod4, k_2 > k_3$ \\
$[k_2, k_2, k_3]$ & $[k_2+1,k_2,k_3]$ & $k_2\text{ odd}, k_3\equiv0\bmod4, k_2 \geq k_3$ \\
$[k_2, k_2, k_3]$ & $0$ & $ k_2 \text{ odd}, k_3 \equiv1\bmod4, k_2 \geq k_3$ \\
$[k_2, k_2, k_3]$ & $[k_2,k_2,k_3+1]$ & $k_2\text{ odd}, k_3\equiv2\bmod4, k_2 > k_3+1$ \\
$[k_3+1,k_3+1,k_3]$ & $3[k_3+1,k_3+1,k_3+1])$ & $ k_3\equiv2\bmod4$\\
$[k_2, k_2, k_3]$ & $[k_2+1,k_2,k_3]+[k_2,k_2,k_3+1]$ & $k_2\text{ odd}, k_3\equiv3\bmod4, k_2 > k_3$\\
$[k_3,k_3,k_3]$ & $[k_3+1,k_3,k_3]$ &$k_3\equiv3\bmod4 $
\end{tabular}
\end{center}
It is easy to see that the  images from the fourth line are twice the images of the seventh line. If one excludes the elements $\{[k_2+1, k_2+1, k_3] | k_2 \equiv 0 \bmod2, k_3 \equiv 3\bmod4\}$ from the basis of $\calC_1$ (i.e. the elements that appear on the fourth line above), then all the non-zero images of such set have different leading terms. Thus these images are linearly independent. To complete the obtained set to a basis of $\calC_0$ we add elements
\begin{align*}
&\{[k_2+1, k_2, k_3] \,|\, k_2 \equiv 0 \bmod2, k_3 \equiv 1 \bmod4, k_2 > k_3\} \cup \\ &\{[k_2+1, k_2, k_3] \,|\, k_2 \equiv 0 \bmod2, k_3 \equiv 2 \bmod4, k_2 > k_3 + 1 \} \cup \{[k_3+1, k_3+1, k_3] \,|\, k_3 \equiv 0 \bmod4\}.
\end{align*}
This immediately implies the above set is a basis of cycles in $H_0$, In other words,
\begin{multline}
H_0 = \frac{\ker(\partial_0)}{\im(\partial_1)} = \langle [k_2+1, k_2, k_3] \,|\, k_2 \equiv 0 \bmod2, k_3 \equiv 1 \bmod4, k_2 > k_3 \rangle \oplus \\ \langle [k_2+1, k_2, k_3] | k_2 \equiv 0 \bmod2, k_3 \equiv 2 \bmod4, k_2 > k_3 \rangle \oplus \langle [k_3+1, k_3+1, k_3] | k_3 \equiv 0 \bmod4 \rangle.
\label{eq_H0_m_odd_n_even}
\end{multline}
We will now compute $H_1$. It is easy to see from above that
\begin{multline*}
\ker(\partial_1) = \im(\partial_2) \oplus \langle 2[k_2+1, k_2+1, k_3] - [k_2+1, k_2, k_3+1] | k_2 \equiv 0 \bmod2, k_3 \equiv 2 \bmod4, k_2 > k_3 \rangle \oplus \\ \langle [k_2, k_2, k_3] | k_2 \equiv 1 \bmod2, k_3 \equiv 1 \bmod4, k_2 \geq k_3 \rangle.
\end{multline*}
Thus we get
\begin{multline}
H_1 = \langle 2[k_2+1, k_2+1, k_3] - [k_2+1, k_2, k_3+1] | k_2 \equiv 0 \bmod2, k_3 \equiv 2 \bmod4, k_2 > k_3 \rangle \oplus \\ \langle [k_2, k_2, k_3] | k_2 \equiv 1 \bmod2, k_3 \equiv 1 \bmod4, k_2 \geq k_3 \rangle.
\label{eq_H1_m_odd_n_even}
\end{multline}

\begin{theorem}\label{t_homology_m_odd_n_even}
In case $m$ is odd and $N$ is even, the generating functions of the two-loop graph-homology $H_0$ and $H_1$, and of the Euler characteristics, are
\begin{align*}
h_0(t) &= \frac{t^2 + t^{11}}{(1-t^4)(1-t^{12})}, \\
h_1(t) &= \frac{t^4 + t^{13}}{(1-t^4)(1-t^{12})}, \\
\chi(t) &= \frac{-t^2  + t^{11}}{(1+t^2)(1-t^{12})}.
\end{align*}
\end{theorem}

\begin{proof}
We start with $h_0$. The first summand in~\eqref{eq_H0_m_odd_n_even} has the generating functions of dimensions
$$
\frac{t^6}{(1-t^4)(1-t^{12})}.
$$
In this formula $t^6$  in the numerator corresponds to $[3, 2, 1]$. The denominator is obtained exactly as described in the proof of Theorem~\ref{t_homology_m_even_n_odd}. The second summand has the generating function of the dimensions
$$
\frac{t^{11}}{(1-t^4)(1-t^{12})},
$$
where $t^{11}$ comes from $[5,4,2]$.
 The last summand of~\eqref{eq_H0_m_odd_n_even} has the generating function of dimensions
$$
\frac{t^2}{1-t^{12}},
$$
where $t^2$ corresponds to $[1, 1, 0]$. \\
\indent For $h_1(t)$, the first summand of~\eqref{eq_H1_m_odd_n_even} has the generating function of dimensions
$$
\frac{t^{13}}{(1-t^4)(1-t^{12})}.
$$
Here the numerator $t^{13}$ corresponds to $2[5, 5, 2] - [5, 4, 3]$. For the second summand we get the function
$$
\frac{t^4}{(1-t^4)(1-t^{12})}.
$$
Here $t^4$  stands for the element $[1, 1, 1]$. \\
\indent To compute $\chi(t)$, again we use Lemma~\ref{l_euler}.
\end{proof}

\section{Concentration of Homology}
We will now examine the concentration of homology in all four cases.

We will first focus on the case where both $N$ and $m$ are odd. Recall that the formula for the graded dimension of $H_0$ is
$$
h_0(t)=\frac{1}{(1-t^2)(1-t^6)}-1
$$

and the graded dimension for $H_1$ is

$$
h_1(t)= \frac{t}{(1-t^2)(1-t^6)}.
$$

\begin{theorem}\label{t_rank_m_odd_n_odd}
Let $a_k$ represent the sequence of coefficients generated by $h_0$, and let $b_k$ represent the sequence of coefficients generated by $h_1$. When $N$ and $m$ are odd, we have
$$
a_k=\begin{cases}
\lceil{\frac{k+1}{6}}\rceil &k \equiv 0 \bmod 2\\
0 &k \equiv 1 \bmod 2\\
\end{cases},
$$
and
$$
b_k=\begin{cases}
0 &k \equiv 0 \bmod 2\\
\lceil{\frac{k}{6}}\rceil &k \equiv 1 \bmod 2\\
\end{cases}.
$$
\end{theorem}
\begin{proof}
Follows from the fact that the expansion of $\frac{1}{(1-t^2)(1-t^6)}$ has coefficient zero in front of any odd exponent $t^k$  and coefficient $\lceil{\frac{k+1}{6}}\rceil$ in front of any even exponent.
\end{proof}

The next case we examine is when both $N$ and $m$ are even. We have
\begin{align*}
h_0(t) &= \frac{t^6}{(1-t^2)(1-t^6)}, \\
h_1(t) &= \frac{t^7}{(1-t^2)(1-t^6)}.\\
\end{align*}

\begin{theorem}\label{t_rank_m_even_n_even}
As before, let $a_k$ represent the sequence of coefficients generated by $h_0$, and let $b_k$ represent the sequence of coefficients generated by $h_1$. Then, when $N$ and $m$ are even,
\begin{align*}
a_k &= \begin{cases}
\lfloor{\frac{k}{6}}\rfloor &k \equiv 0 \bmod 2 \\
0 & k \equiv 1 \bmod 2\\
\end{cases}\\
b_k &= \begin{cases}
0 &k \equiv 0 \bmod 2\\
\lfloor{\frac{k}{6}}\rfloor &k \equiv 1 \bmod 2\\
\end{cases}
\end{align*}
\end{theorem}
\begin{proof}
Similar to the previous theorem.
\end{proof}

We now turn to the cases where $N+m$ is odd, beginning with $N$ odd and $m$ even.  Our formulas for graded dimension are

$$
h_0(t) = \frac{t^3 + t^{11} + t^{14} - t^{15}}{(1-t^4)(1-t^{12})},
$$
and the graded dimension for $H_1$ is
$$
h_1(t) = \frac{t + t^{16}}{(1-t^4)(1-t^{12})}.
$$

\begin{theorem}\label{t_rank_m_even_n_odd}
When $N$ is odd and $m$ is even, using the same notation as above, we have
$$
a_k=\begin{cases}
0 &k\equiv1,0 \bmod4\\
\lfloor{\frac{k}{12}}\rfloor &k\equiv 2 \bmod4\\
\lceil{\frac{k+2}{12}}\rceil  &k\equiv 3 \bmod 4\\
\end{cases}
$$
and
$$
b_k=\begin{cases}
0 &k\equiv2,3 \bmod4\\
\lfloor{\frac{k-1}{12}}\rfloor &k\equiv 0 \bmod4\\
\lceil{\frac{k}{12}}\rceil  &k\equiv 1 \bmod 4\\
\end{cases}
$$
\end{theorem}
\begin{proof}
Follows from the fact that $\frac 1{(1-t^4)(1-t^{12})}$ has coefficient zero in front of any exponent $t^k$ with $k$ non-divisible by~4, and coefficient $\lceil{\frac{k+1}{12}}\rceil$ for $k$ divisible by~4.
\end{proof}

Finally, when $N$ is even and $m$ is odd, recall that the graded dimension of $H_0$ is
$$
h_0(t) = \frac{t^2+t^{11}}{(1-t^4)(1-t^{12})},
$$
and the graded dimension for $H_1$ is
$$
h_1(t) = \frac{t^4+t^{13}}{(1-t^4)(1-t^{12})}.
$$
Again, we expand these series with respect to $t$ to find a general formula to compute the coefficient for a given exponent.

\begin{theorem}\label{t_rank_m_odd_n_even}
When $N$ is even and $m$ is odd, we have
$$
a_k=\begin{cases}
0 &k\equiv1,0 \bmod4\\
\lfloor{\frac{k+1}{12}}\rfloor &k\equiv 3 \bmod4\\
\lceil{\frac{k}{12}}\rceil  &k\equiv 2 \bmod 4\\
\end{cases},
$$
where $a_k$ represents the coefficient of $H_0$, and
$$
b_k=\begin{cases}
0 &k\equiv2,3 \bmod4\\
\lfloor{\frac{k}{12}}\rfloor &k\equiv1 \bmod4\\
\lceil{\frac{k}{12}}\rceil  &k\equiv 0 \bmod 4\\
\end{cases}.
$$
where $b_k$ represents the coefficient of $H_1$.
\end{theorem}

\begin{proof}
Similar to the previous theorem.
\end{proof}

\begin{remark}
Notice that, for a given Hodge grading, in all the four cases the homology is concentrated in a single degree.
\end{remark}


\section*{Acknowledgements}
All the authors acknowledge NSF for the support. The idea for this work appeared during Conant's visit to K-State in 2009  supported by the Midwest Topology Network grant DMS~0844249. At that moment Conant and Turchin made computations in the \lq\lq easy case" of $m+N$ even. Later on Costello and Weed were working with Turchin on this project supported by the summer REU program DMS~1004336. Costello and Weed redid the computation that were previously done and finished them by computing the \lq\lq difficult case" of $m+N$ odd. Costello and Weed thank K-State for hospitality and also Baltazar Chavez-Diaz -- another REU student, for discussions and his contribution to the project.  Finally Turchin was also partially supported   by DMS~0967649 and by the Max Planck Institute for Mathematics in Bonn, where he was working on the last version of this paper, and which he thanks  for hospitality.

\appendix

\section*{Tables}

In this appendix section we present tables of ranks of the 2-loops graph-homology in small Hodge degrees. Notice that the results of our computations of Euler characteristics (the third column in each table) confirm previous computations --- the diagonal  $s=t+1$ in~\cite[Tables~1,~3,~5,~7]{AT2}.

\begin{table}[h!]
\parbox{.45\linewidth}{
\centering
\begin{tabular}{|c|c|c|c|}
\hline
Hodge Degree & $h_0(t)$ & $h_1(t)$ & $\chi(t)$ \\ \hline
1&0&1&-1 \\ \hline
2&1&0&1 \\ \hline
3&0&1&-1 \\ \hline
4&1&0&1 \\ \hline
5&0&1&-1 \\ \hline
6&2&0&2 \\ \hline
7&0&2&-2 \\ \hline
8&2&0&2 \\ \hline
9&0&2&-2 \\ \hline
10&2&0&2 \\ \hline
11&0&2&-2 \\ \hline
12&3&0&3 \\ \hline
13&0&3&-3 \\ \hline
14&3&0&3 \\ \hline
15&0&3&-3 \\ \hline
16&3&0&3 \\ \hline
17&0&3&-3 \\ \hline
18&4&0&4 \\ \hline
19&0&4&-4 \\ \hline
20&4&0&4 \\ \hline
21&0&4&-4 \\ \hline
22&4&0&4 \\ \hline
23&0&4&-4 \\ \hline
\end{tabular}
\caption{Concentration of homology when $N$ and $m$ are odd.}
}
\hfill
\parbox{.45\linewidth}{
\centering
\begin{tabular}{|c|c|c|c|}
\hline
Hodge Degree & $h_0(t)$ & $h_1(t)$ & $\chi(t)$ \\ \hline
1&0&0&0 \\ \hline
2&0&0&0 \\ \hline
3&0&0&0 \\ \hline
4&0&0&0 \\ \hline
5&0&0&0 \\ \hline
6&1&0&-1 \\ \hline
7&0&1&1 \\ \hline
8&1&0&-1 \\ \hline
9&0&1&1 \\ \hline
10&1&0&-1 \\ \hline
11&0&1&1 \\ \hline
12&2&0&-2 \\ \hline
13&0&2&2 \\ \hline
14&2&0&-2 \\ \hline
15&0&2&2 \\ \hline
16&2&0&-2 \\ \hline
17&0&2&2 \\ \hline
18&3&0&-3 \\ \hline
19&0&3&3 \\ \hline
20&3&0&-3 \\ \hline
21&0&3&3 \\ \hline
22&3&0&-3 \\ \hline
23&0&3&3 \\ \hline
\end{tabular}
\caption{Concentration of homology when $N$ and $m$ are even.}
}
\end{table}

\begin{table}[h!]
\parbox{.45\linewidth}{
\centering
\begin{tabular}{|c|c|c|c|}
\hline
Hodge Degree& $h_0(t)$ & $h_1(t)$ & $\chi(t)$ \\ \hline
1&0&1&1 \\  \hline
2&0&0&0 \\ \hline
3&1&0&-1 \\ \hline
4&0&0& 0\\ \hline
5&0&1&1 \\ \hline
6&0&0&0 \\ \hline
7&1&0&-1 \\ \hline
8&0&0&0 \\ \hline
9&0&1&1 \\ \hline
10&0&0&0 \\ \hline
11&2&0&-2 \\ \hline
12&0&0&0 \\ \hline
13&0&2&2 \\ \hline
14&1&0&1 \\ \hline
15&2&0&-2 \\ \hline
16&0&1&-1 \\ \hline
17&0&2&2 \\ \hline
18&1&0&1 \\ \hline
19&2&0&-2 \\ \hline
20&0&1&-1 \\ \hline
21&0&2&2 \\ \hline
22&1&0&1 \\ \hline
23&3&0&-3 \\
\hline
\end{tabular}
\caption{Concentration of homology when $N$ is odd and $m$ is even.}
}
\hfill
\parbox{.45\linewidth}{
\centering
\begin{tabular}{|c|c|c|c|}
\hline
Hodge Degree & $h_0(t)$ & $h_1(t)$ & $\chi(t)$ \\ \hline
1&0&0&0 \\ \hline
2&1&0&-1 \\ \hline
3&0&0&0 \\ \hline
4&0&1&1 \\ \hline
5&0&0&0 \\ \hline
6&1&0&-1 \\ \hline
7&0&0&0 \\ \hline
8&0&1&1 \\ \hline
9&0&0&0 \\ \hline
10&1&0&-1 \\ \hline
11&1&0&1 \\ \hline
12&0&1&1 \\ \hline
13&0&1&-1 \\ \hline
14&2&0&-2 \\ \hline
15&1&0&1 \\ \hline
16&0&2&2 \\ \hline
17&0&1&-1 \\ \hline
18&2&0&-2 \\ \hline
19&1&0&1 \\ \hline
20&0&2&2 \\ \hline
21&0&1&-1 \\ \hline
22&2&0&-2 \\ \hline
23&0&2&2 \\
\hline
\end{tabular}
\caption{Concentration of homology when $N$ is even and $m$ is odd.}
}
\end{table}
\end{document}